\documentclass[12pt,reqno]{amsart}
\usepackage{amsfonts}
\usepackage{lineno}
\usepackage{graphicx,subfigure}
\usepackage{amssymb}
\usepackage[alphabetic]{amsrefs}
\usepackage{epstopdf}
\usepackage{geometry}
\usepackage{float}
\usepackage{amsmath}
\usepackage{fullpage}
\usepackage{enumerate}
\usepackage{color}

\theoremstyle{plain}
\newtheorem{theorem}{Theorem}[section]
\newtheorem{lemma}[theorem]{Lemma}
\newtheorem{corollary}[theorem]{Corollary}

\newtheorem{question}[theorem]{Question}

\theoremstyle{definition}
\newtheorem{definition}[theorem]{Definition}
\newtheorem{example}[theorem]{Example}
\newtheorem{remark}[theorem]{Remark}
\newcommand{\ds}{\displaystyle}

\def\RR{\mathbb{R}}

\definecolor{darkpastelgreen}{rgb}{0.01, 0.75, 0.24}
\begin{document}

\title{CONTRACTIBLE 3-MANIFOLDS  AND  
THE DOUBLE 3-SPACE PROPERTY}

\author{Dennis J. Garity}
\address{Mathematics Department, Oregon State University,
Corvallis, OR 97331, U.S.A.}
\email{garity@math.oregonstate.edu}
\urladdr{http://www.math.oregonstate.edu/\symbol{126}garity}

\author{Du\v{s}an D. Repov\v{s}}
\address{%
Faculty of Education,
and Faculty of Mathematics and Physics,
University of Ljubljana\\
Ljubljana, SI-1000, Slovenia}
\email{dusan.repovs@guest.arnes.si}
\urladdr{http://www.fmf.uni-lj.si/\symbol{126}repovs}

\author{David G. Wright}
\address{Department of Mathematics, Brigham Young University,
Provo, UT 84602, U.S.A.}
\email{wright@math.byu.edu}
\urladdr{http://www.math.byu.edu/\symbol{126}wright}

\date{\today}



\begin{abstract}
Gabai  showed that the Whitehead manifold is the union of two submanifolds each of which is homeomorphic to $\mathbb R^3$ and whose intersection is again homeomorphic to $\mathbb R^3$.  Using a family of generalizations of the Whitehead Link, we show that there are uncountably many contractible 3-manifolds with this double 3-space property.  Using a separate family of generalizations of the Whitehead Link and using an extension of interlacing theory, we also show that there are uncountably many contractible 3-manifolds that fail to have this property.

\end{abstract}
\maketitle

\section{Introduction}\label{IntroSec}
Gabai \cite{Gab11} showed a surprising result that the  Whitehead contractible 3-manifold \cite{Whi35} is the union of two sub-manifolds each of which is homeomorphic to Euclidean 3-space $\mathbb R^3$ and whose intersection is also homeomorphic to $\mathbb R^3$.  A 3-manifold with this \emph{double 3-space property} must be a contractible open 3-manifold.  The manifold $\mathbb R^3$ clearly has this property, but it takes a lot of ingenuity to show that the Whitehead contractible 3-manifold has the double 3-space property.  This naturally raises two questions: 
\begin{enumerate}
\item[1)]  Are there other contractible 3-manifolds with this property? 
\item[2)]  Do all contractible 3-manifolds have this property?  
\end{enumerate}
We show the answer to the first question is yes by constructing uncountably many contractible 3-manifolds with the double 3-space property.  We show the answer to the second question is no by constructing uncountably many contractible 3-manifolds that fail to have the double 3-space property. The answer to the first question requires careful use of geometric index and a detailed analysis of a family of links we call Gabai Links. The answer to the second question requires a careful extension of interlacing theory originally introduced in \cite{Wri89} and \cite{AS89}.

\section{Definitions and Preliminaries}\label{RefSec}  
A solid torus is homeomorphic to 
 $B^2 \times S^1$ where $B^2$ is a 2-dimensional disk and $S^1$ is a circle.  All spaces and embeddings will be piecewise-linear \cite{RS82}.  If $M$ is a manifold with boundary, then ${\rm Int}\ M$ denotes the interior of $M$ and $\partial\, M$ denotes the boundary of $M$.  We let $\RR^3$ denote Euclidean 3 -space. A \emph{disk with holes} is a compact, connected planar 2-manifold with boundary.  A properly embedded disk with holes $H$ in a solid torus $T$ is said to be \emph{interior-inessential} if the inclusion map on $\partial H$ can be extended to a map of $H$ into $\partial\,T$. If the inclusion map on $\partial H $ cannot be extended to a map of $H$ into $\partial\, T$ we say that $H$ is \emph{interior-essential}  \cite{Dav07}, \cite[p. 170]{DV09}.  If $H$ is interior-essential, we also say $H$ is a \emph{meridional disk with holes} for the solid torus $T$. See Figure \ref{diskwHolesFig} for an illustration of a meridional disk with holes with  five boundary components.

For background on contractible open 3-manifolds, see \cite{McM62, Mye88, Mye99a, Wri92}.

\begin{definition} 
A \emph{Whitehead Link} is a pair of solid tori $T' \subset {\rm Int}\ T$  so that $T'$  is contained in ${\rm Int}\ T$ as illustrated in  Figure \ref{GabaiLinksFig:Whitehead}.
\end{definition}

\begin{figure}[ht!]
\begin{center}
  \subfigure[Whitehead Link]%
    {%
    \label{GabaiLinksFig:Whitehead}
    \includegraphics[width=0.48\textwidth]{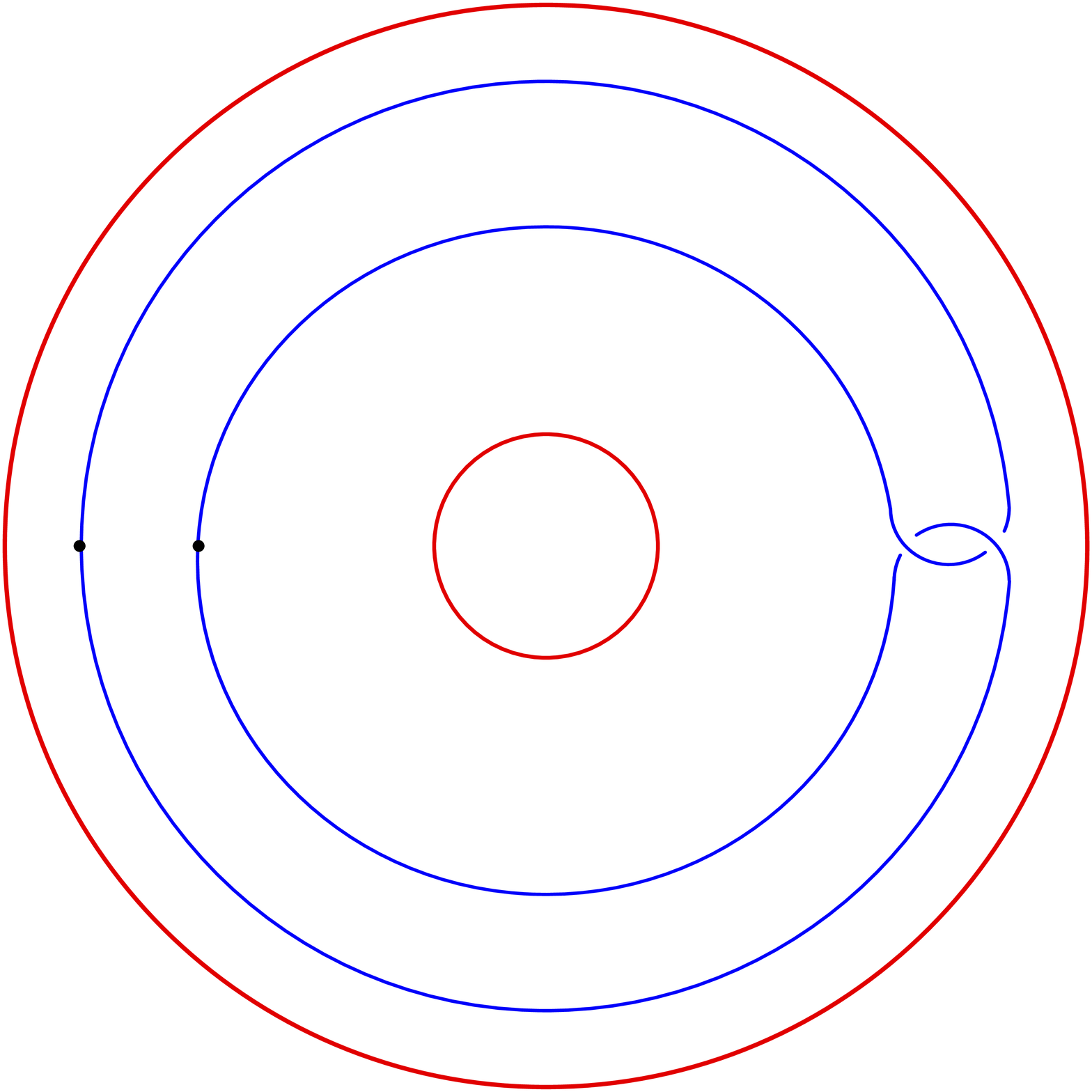}
    }%
  \subfigure[Gabai Link: Order 2]%
    {%
    \label{GabaiLinksFig:Order2}
    \includegraphics[width=0.48\textwidth]{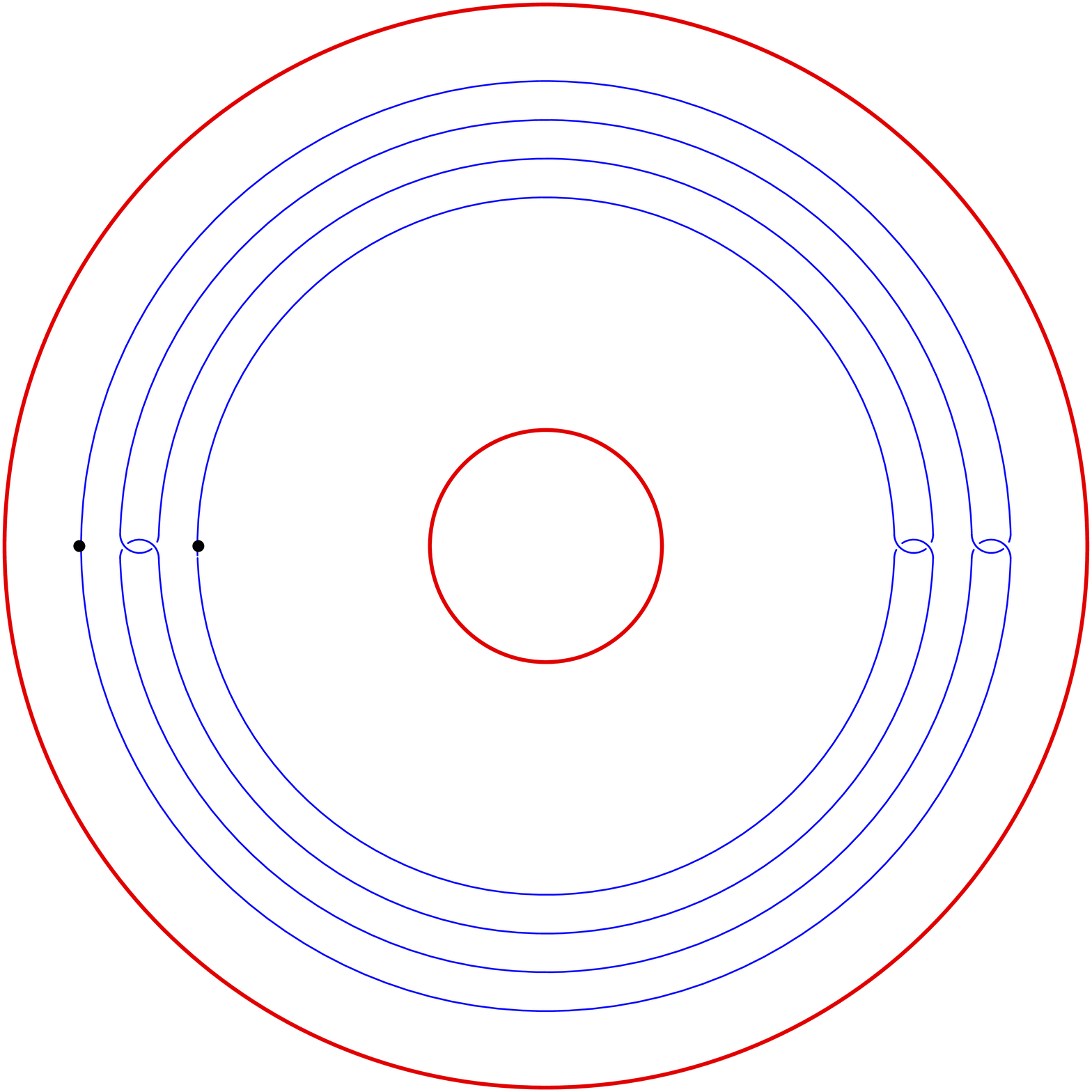}
    }\\%
  \subfigure[Gabai Link: Order 3]%
    {%
    \label{GabaiLinksFig:Order3}
    \includegraphics[width=0.48\textwidth]{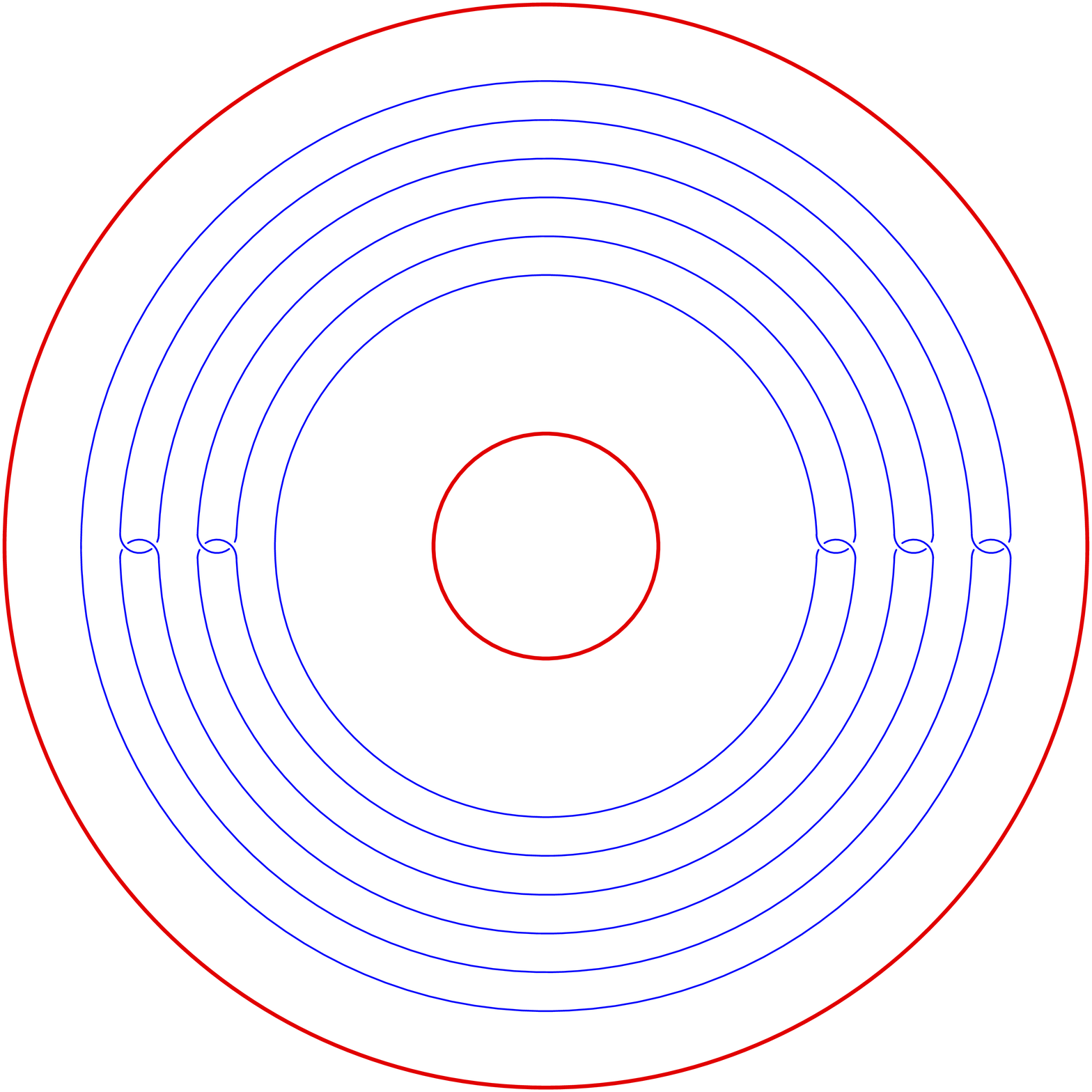}
    }%
  \subfigure[Gabai Link: Order $n$]%
    {%
    \label{GabaiLinksFig:Ordern}
    \includegraphics[width=0.48\textwidth]{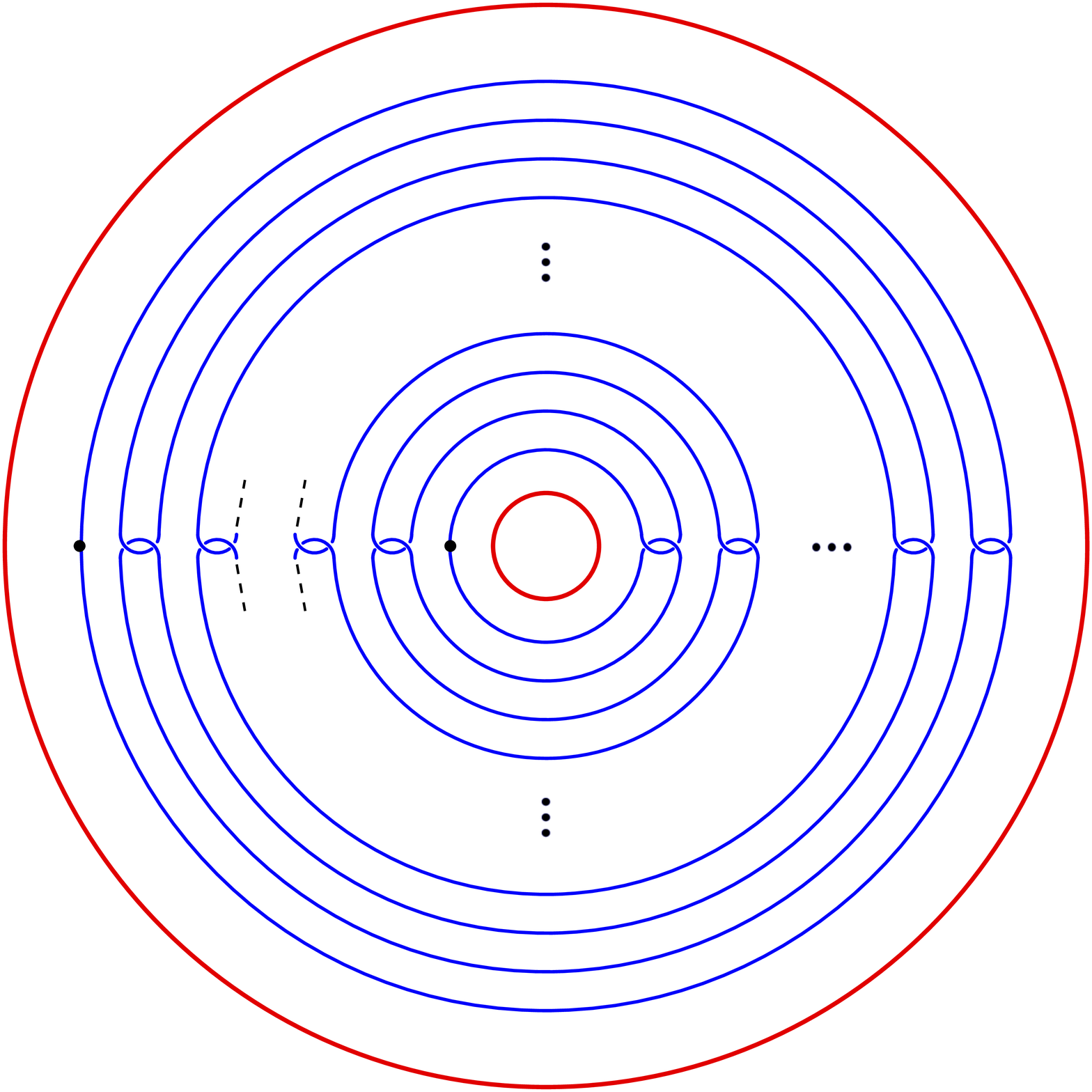}
    }%
  \end{center}
  \caption{%
    Whitehead and Gabai Links}%
  \label{GabaiLinksFig}
\end{figure}

The famous Whitehead contractible 3-manifold  \cite{Whi35}  is a  3-manifold  that is the ascending union of nested solid tori $T_i,i\geq 0,$ so that for each $i$, $T_i \subset {\rm Int}\ T_{i+1}$ is a Whitehead Link.  

\begin{definition} \label{geometric}
If $T' \subset {\rm Int}\ T$ are solid tori, the \emph{geometric index} of $T'$ in $T$, $N(T',T)$, is the minimal number of points of the intersection of the centerline of $T'$ with a meridional disk of $T$. 
\end{definition}     

See Schubert \cite{Sch53} and \cite{GRWZ11} for the following  results about geometric index.
\begin{itemize}
\item Let $T_0$ and $T_1$ be unknotted solid tori in $S^{3}$ with  $T_0 \subset {\rm Int}\ T_1$ and $N( T_0, T_1) = 1$.  Then $\partial\, T_0$ and  $\partial\, T_1$ are parallel.
\item   Let $T_0, T_1$, and $T_2$ be solid tori so that $T_0 \subset {\rm Int}\ T_1$ and $T_1 \subset {\rm Int}\ T_2$.  Then $N(T_0, T_2) =  N(T_0, T_1) \cdot  N(T_1, T_2)$.
\item If  $T_0$ and $T_1$ are unknotted solid tori in $S^{3}$ with  $T_0 \subset {\rm Int}\ T_1$, and if $T_0$ is contractible in $T_1$, then $N( T_0, T_1)$ is even.
\end{itemize}

\begin{remark}\label{WH-index} 
If $T' \subset {\rm Int}\ T$ is a Whitehead Link, then the geometric index of $T'$ in $T$ is 2. This follows since the geometric index is 0 or 2 by the third bullet above. It cannot be 0 by \cite[Section 3.G.4]{Rol90}. 
\end{remark}

We now define a generalization of the Whitehead Link (which  has geometric index 2) to a \emph{Gabai Link} that will be shown to have geometric index $2n$ for some positive integer $n$. We will use this generalization in Section \ref{GabaiSec} to produce our examples of 3-manifolds that have the double 3-space property.

\begin{definition}\label{GabaiLinkDef}
  Let $n$ be a positive integer.  A \emph{Gabai Link} of order  $n$ is a pair of solid tori  $T' \subset {\rm Int}\ T$ as illustrated in Figure  \ref{GabaiLinksFig}. Figure \ref{GabaiLinksFig:Order2}  shows a Gabai Link of order 2, Figure  \ref{GabaiLinksFig:Order3} shows a Gabai Link of order 3, and Figure  \ref{GabaiLinksFig:Ordern} shows a generalized Gabai Link of order $n$.  For the link of  order $n$, there are $n-1$ clasps on the left and $n$ clasps on the right.
\end{definition}

\textbf{Note}: In a Gabai Link of order $n$ the inner torus has geometric index $2n$ in the outer torus. See  Lemma \ref{G index 3} in the Appendix for details on this.

Note that the inner torus $T'$ in a Gabai Link is contractible in the outer torus $T$.

\begin{definition}  A \emph{genus one 3-manifold} $M$ is the ascending union of solid tori $T_i, i\geq 0$, so that for each $i$,  $T_i \subset {\rm Int}\  T_{i+1} $ and the geometric index of $T_i$ in $T_{i+1}$ is not equal to $0$.
\end{definition}

\begin{theorem}  
If $M$ is a genus one 3-manifold with defining sequence $(T_i)$, then, for each j,  $T_j$ does not lie in any open subset of $M$ that is homeomorphic to $\mathbb R^3$.
\label{GenusOneTheorem}
\end{theorem}

{\sc Proof.}  
If $T_j$ lies in $U$ so that $U$ is homeomorphic to $\mathbb R^3$, then, since $T_j$ is compact, it lies in a 3-ball $B \subset U$.  Since $B$ is compact, it lies in the interior of some $T_k$ with $k > j$.  This implies that the geometric index of $T_j$ in $T_k$ is 0, but since the geometric index is multiplicative, the geometric index of $T_j$ in $T_k$ is not zero.  So there is no such $U$.
\qed

\begin{theorem} \label{essential} 
If $M$ is a genus one 3-manifold with defining sequence $(T_i)$, and $J$ is an essential simple closed curve that lies in some $T_j$, then $J$ does not lie in any open subset of $M$ that is homeomorphic to $\mathbb R^3$.
\end{theorem}

{\sc Proof.}  
By thickening up $T_j$ we may assume, without loss of generality, that $J$ is the centerline of a solid torus $T$ that lies in  ${\rm Int}\ T_j$. Since $J$ is essential in $T_j$,  the geometric index of $T$ in $T_j$ is not equal to zero.  Thus, $M$ is the ascending union of tori $T, T_j, T_{j+1}, T_{j+2}, \cdots$ and by the previous theorem, $T$ does not lie in any open subset of $M$ that is homeomorphic to $\mathbb R^3$.  If $J$  lies in $U$ so that $U$ is homeomorphic to $\mathbb R^3$, then we could have chosen $T$ so that it also lies in $U$.  Thus,  by Theorem \ref{GenusOneTheorem}, $J$ does not lie in any open subset of $M$ that is homeomorphic to $\mathbb R^3$.
\qed

\begin{theorem}  
A genus one 3-manifold $M$ with defining sequence $(T_i)$ so that each $T_i$ is contractible in $T_{i+1}$,  is a contractible 3-manifold that is not homeomorphic to $\mathbb R^3$.
\end{theorem}

{\sc Proof.}  
$M$ is contractible since all the homotopy groups are trivial. 
This is a consequence of the Whitehead Theorem. See for example \cite[Section 0.9]{DV09}.

If $M$ is homeomorphic to $\mathbb R^3$, then each $T_i$ in the defining sequence lies in an open subset that is homeomorphic to $\mathbb R^3$ which is a contradiction by Theorem \ref{GenusOneTheorem}.
\qed

\begin{definition}  
A 3-manifold is said to satisfy  the \emph{double 3-space property} if it is the union of two open sets $U$ and $V$ so that each of $U$, $V$, and $U \cap V$ is homeomorphic to $\RR^3$.
\end{definition}

\section{Gabai Manifolds Satisfy the Double 3-space Property}\label{GabaiSec}
\subsection {Gabai Manifolds}
Refer to Definition \ref{GabaiLinkDef} and Figure \ref{GabaiLinksFig} for the definition of a Gabai Link.

\begin{definition}\label{GabaiManifoldDef}
A \emph{Gabai contractible 3-manifold} is the ascending union of nested solid tori so that any two consecutive tori form a Gabai Link.  Given a sequence $n_1, n_2, n_3, \dots$ of  positive integers, there is a Gabai contractible 3-manifold $G = \ds\bigcup_{m = 0}^\infty T_m$ so that the tori $T_{m-1}
\subset {\rm Int}\  T_m$ form a Gabai Link of  order $n_m$.  
\end{definition}

In fact, it is possible to assume that each $T_m \subset \mathbb R^3$ because if a Gabai Link is embedded in $\RR^3$ so that the larger solid torus is unknotted, then the smaller solid torus is also unknotted.  McMillan's proof \cite{McM62}  that there are uncountably many genus one contractible 3-manifolds transfers immediately to show that there are uncountably many Gabai contractible 3-manifolds. This proof uses properties of geometric index to show that if a prime $p$ is a factor of infinitely many of $n_1, n_2, n_3, \dots$ and only finitely many of $m_1, m_2, m_3, \dots$, then the two 3-manifolds formed using these sequences cannot be homeomorphic.

\subsection{Special Subsets of $S^1$ and $B^2\times S^1$}
\label{SpecialSubsets}
In $S^1$ choose a closed interval $I$ which we identify with the closed interval $[0,1]$.  Let $C \subset I\subset S^1$ be the standard middle thirds Cantor set.  Let $U_1 = ({1 \over 3}, {2 \over 3})$, $U_2 = ({1 \over 9}, {2 \over 9}) \cup ({7 \over 9}, {8 \over 9}) $, and, in general, $U_i$ be the union of the $2^{i-1}$ components of $[0,1] - C$ that have length $1/3^i$.  Let $U_0 = S^1 - [0,1]$, $C_1 = C \cap [0, {1 \over 3}]$, and $C_2  = C \cap [{2 \over 3},1]$.

Let $h:B^2\times S^1 \to \RR^3$ be an embedding so that  $T=h( B^2\times S^1)$ is a standard unknotted solid torus in $\RR^3$. Set $V^i = h(B^2 \times U_i$),   $A = h(B^2 \times C_1)$, and  $B = h(B^2 \times C_2)$.  So  $V^i$ (for $i \ge 0$), $A$, and $B$ are all subsets of $T$.  The subset $V^0$ is homeomorphic to $B^2 \times (0,1)$.  For $i >0$,   $V^i $ is homeomorphic to the disjoint union of $2^{i-1}$ copies of $B^2 \times (0,1)$, and both $A$ and $B$ are homeomorphic to $B^2 \times C$.

For each positive integer $n$, let $g_n$ be a homeomorphism of $\RR^3$ to $\RR^3$ that takes $T$ into its interior, so that the pair $(g_n(T), T)$ forms a Gabai Link of geometric index $2n$. Let $T^{\,\prime}_n=g_n(T)$, $A^{\,\prime}_n=g_n(A)$,  $B^{\,\prime}_n=g_n(B)$, and ${V^{i}_n}^{\,\,\prime}=g_n(V^i)$.

\begin{lemma}\label{SetupLemma}
The homeomorphisms $g_n:\RR^3\rightarrow \RR^3$ can be chosen  so that:
\begin{subequations}\label{SetupConditions}
\begin{eqnarray}
&A \cap T^{\,\prime}_n = A^{\,\prime}_n \text{ and } B \cap T^{\,\prime}_n = B^{\,\prime}_n \label{SetupConditions1}\\
&{V^{0}_n}^{\,\prime} \subset V^0, \text{ and }\label{SetupConditions2}\\
&
\text{ for } i>0, \text{ and for each component } X \text { of } {V^{i}_n}^{\,\prime}
 \text{ there is }  j <i \text{ such that } X\subset V^j
\label{SetupConditions3}
\end{eqnarray}
\end{subequations} 
\end{lemma}

{\sc Proof.}  
Fix a positive integer $n$. We first define $g_n$ on $T$. The idea is to identify $4n$ subsets of $T=B^2\times S^1$, each homeomorphic to a tube of the form an interval cross $B^2$,  and to use the $S^1$ coordinate to linearly (in the $S^1-U_0=I$ factor) stretch these tubes from the region $V^0$ to the region $V^1$ in $T$.
 
Choose a positive integer $m$ and a nonnegative integer $k<2^{m-1}$ so that 
$2^{m}+2k=4n<2^{m+1}$. Note that $m$ and $k$ are determined by $n$. Remove the subsets $U_1,\ldots U_m$ from $I$ so that $2^{m}$ intervals of length $3^{-m}$ remain. Then remove $2k$ of the intervals in $U_{m+1}$, namely the middle third of the first $k$ and the last $k$ of these remaining intervals in $I$ so that $4n=2^{m}+2k$ intervals remain, $4k$ of length $3^{-(m+1)}$, 
and the remaining $2^{m}-2k$ of length $3^{-m}$. Let $\widetilde{U}_{m+1}$ be the union of the intervals in $U_{m+1}$ that have been removed. Figure \ref{CantorLabelsFig} shows the case where $n=3, m=3, \text{ and }k=2$. 
The integers $i$ across the bottom of this figure then correspond to  $U_{1},\, U_{2},\, U_3,\, $ and $\widetilde{U}_{4}$ defined above.

\begin{center}
\begin{figure}[ht]
\includegraphics[width=0.85\textwidth]{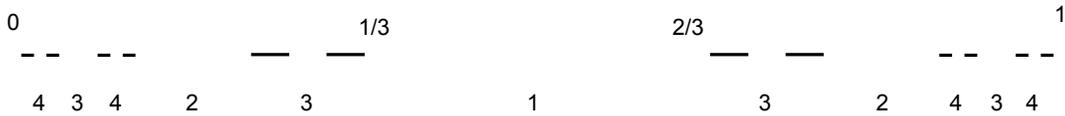}
\caption{%
Labelled Removed Intervals in $[0,1]$     }%
\label{CantorLabelsFig}
\end{figure}
    \end{center}
    
Now let $\widetilde{V}_{m+1}=h(B^2 \times \widetilde{U}_{m+1})$ and consider $W=T-\cup_{j=0}^m V^j - 
\widetilde{V}_{m+1}$. Then $W$ consists of $4n$ tubes homeomorphic to $[0,1]\times B^2$.
Let $g_n$ be a homeomorphism of $T$ into its interior so that:
\begin{enumerate}
\item The pair $(g_n(T), T)$ forms a Gabai Link of geometric index $2n$,
\item  $g_n(V^0\cup V^1)\subset V^0$, 
\item The  components of 
$\cup_{j=2}^m V^j \cup  \widetilde{V}_{m+1}$ are taken by $g_n$ into $V_0 \cup V_1$, and 
\item $g_n$ restricted to each of the $4n$ tubes mentioned above is a product of a homeomorphism of the $B^2$ factor onto a subdisk with a linear homeomorphism on the interval factor that  stretches the tube from $V^0$ to $V^1$ or from $V^1$ to $V^0$ in either $B^2\times[0, 1/3]$ or in $B^2\times[2/3, 1]$.  
\end{enumerate}

Figure \ref{ToriLabelsFig} illustrates this when $n=3$, with the numbers $j$ listed by parts of the interior torus corresponding to the subsets $g_n(V^j)$. The last two regions mentioned in (4) above correspond to the top or bottom components of  $T-(V_0\cup V_1)$ in Figure  \ref{ToriLabelsFig}. In particular, the $S^1$ factor, after $U_0$ is removed, is parameterized in a counterclockwise manner in Figure  \ref{ToriLabelsFig}.

\begin{center}
\begin{figure}[ht]
 \includegraphics[width=0.59\textwidth]{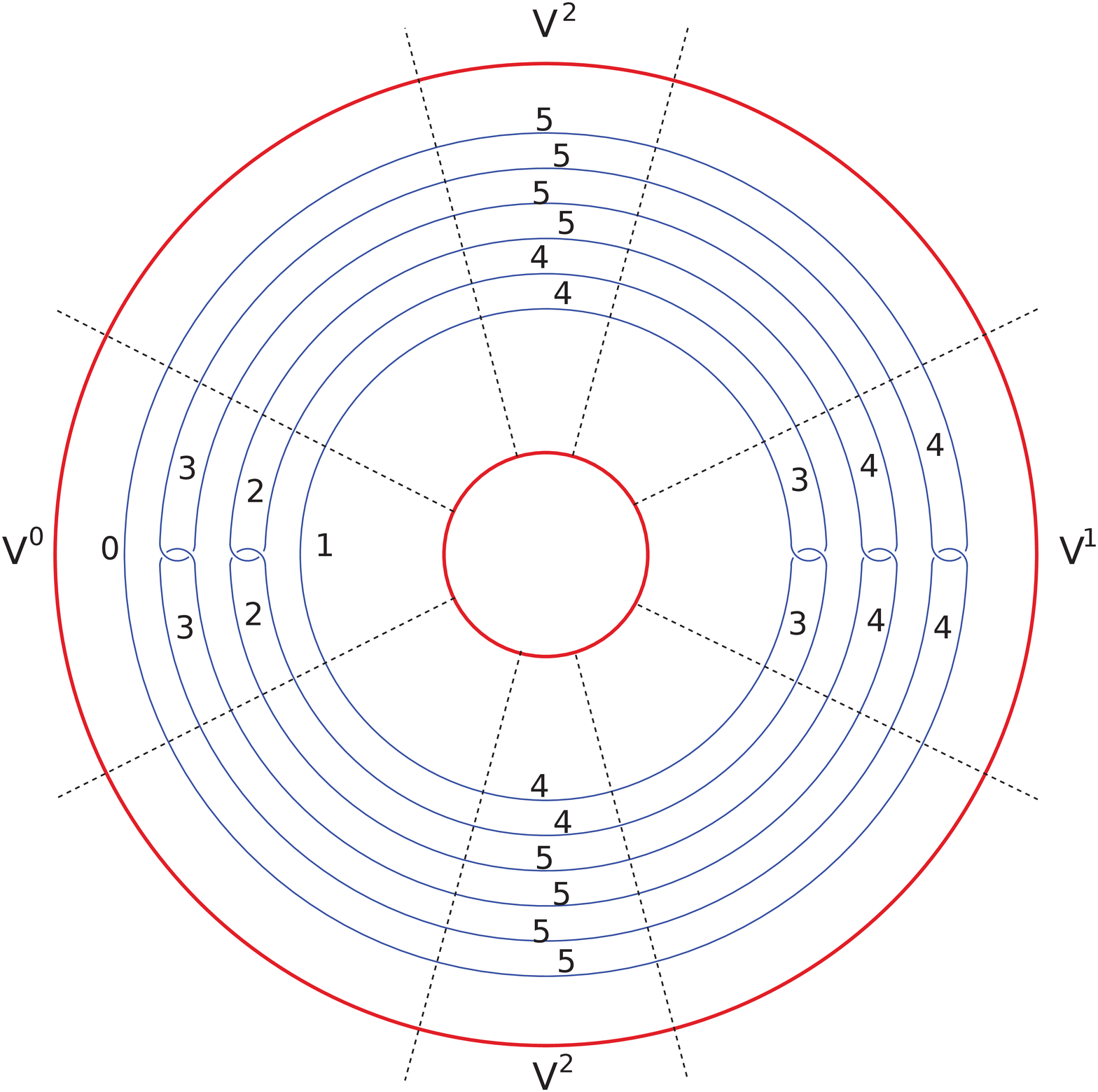}
\caption{%
Labelled Regions on Tori in Gabai Link     }%
\label{ToriLabelsFig}
\end{figure}
    \end{center}

The idea here is that each of the $4n$ tubes of $W$  corresponds in the $S^1$ factor to a remaining interval at stage $m$ or $m+1$ of the Cantor set construction. Stretching each tube out linearly (in the $S^1$ factor) makes the middle third of the interval in the $S^1$ factor match up exactly with 
one of the middle thirds labeled $V^2$ in $T$. The self similarity under linear scaling of the Cantor set also makes further stages of the Cantor set in the $S^1$ factor of the tube match up with further stages of the Cantor set in the $S^1$ factor of $T$.

The interval factor of each of the tubes in $W$ corresponds to an interval in $I$ of length $3^{-m}$ or of length $3^{-(m+1)}$, one of the remaining intervals in stage $m$ or stage $m+1$  of the standard construction of $C$. Let $D$ be one of these intervals. The self similarity of $C$ shows that a linear homeomorphism from $D$ onto either $E=[0,1/3]$ or onto $E=[2/3,1]$ takes $C\cap D$ onto $C\cap E$ and takes the intervals of $U_i\cap D$ homeomorphically to the intervals of $U_{i-\ell}\cap E$ where $\ell=m-1$ or $\ell=m$.

From this, it follows that condition (\ref{SetupConditions2}) is satisfied. The nature of a Gabai Link guarantees that $A_n^{\,\prime}\subset B^2\times[0,1/3]\subset T$ and that $B_n^{\,\prime}\subset B^2\times[2/3,1]\subset T$. This, together with the discussion in the previous paragraph shows that condition (\ref{SetupConditions1}) is satisfied.

Since both $T$ and $T'=g_n(T)$ are unknotted solid tori, the map $g_n$ extends to a homeomorphism of $\RR^3$ if and only if $g_n$ takes a longitudinal curve of $T$ to a
longitudinal curve of $T'$.
By a \emph{longitudinal curve} we mean a simple closed curve in the boundary of $T$ (respectively $T^\prime$) that is ambiently isotopic in $T$ (respectively $T^\prime$) to $S^1\times 
 \{ p\}$ in $S^1\times B^2=T (\text{ respectively } T^\prime)$ where $p$ is a point in the boundary of $B^2$.
If this is not the case, we can first take a twisting
homeomorphism of $T$ to itself that preserves the subsets $A, B$, and $V^i$ of $T$ so
that the compositions of the twisting homeomorphism and our $g_n$ takes a longitudinal
curve of $T$ to a longitudinal curve of $T'$. Thus we may assume that $g_n$ extends to a
homeomorphism of $\RR^3$ to itself. 
\qed

\subsection{Construction}
We will now inductively construct a Gabai 3-manifold corresponding to a sequence $n_1, n_2, n_3, \dots$ of  positive integers, with special subsets corresponding to the subsets of $T$ 
and $T_n^{\prime}$ just described.  Let $T_0=T$. Let $h_1:T\rightarrow R^3$ be given by 
$g_{n_1}^{-1}$ and let $T_1=h_1(T)$. Let $A_1=h_1(A)$,  $B_1=h_1(B)$, and $V^{i}_1=h_1(V^i)$. Note that the pair $(T_0,T_1)$ is homeomorphic to $(T^{\prime}_{n_1},T)$ via $g_{n_1}$and so forms a Gabai Link of  order $n_1$. It follows immediately  from Lemma \ref{SetupLemma} and the definitions of the various subsets that:

\begin{subequations}\label{Case1}
\begin{eqnarray}
&A_1 \cap T_0 = A \text{ and }B_1 \cap T_0 = B \label{Case1:ab}\\
&V^{0} \subset V_{1}^0 
\text{ and for } i>0,  V^{i} \subset  \cup_{j<i}V_{1}^j
\label{Case1:cd}
\end{eqnarray}
\end{subequations} 

Inductively assume that homeomorphisms $h_i:T\rightarrow R^3$ have been described for $i\leq  k$ and that $A_i=h_i(A)$,  $B_i=h_i(B)$, and $V^{j}_i=h_i(V^j)$ for $i\leq k$.  
Let $T_i=h_i(T)$.
Also assume that for each $i\leq k$:

\begin{subequations}\label{Inductive}
\begin{eqnarray}
&A_i \cap T_{i-1} = A_{i-1}
\text{ and }B_i \cap T_{i-1} = B_{i-1} \label{Inductive:ab}\\
&V^{0}_{i-1} \subset V^{0}_i 
\text{ and for } j>0,  V^{j}_{i-1} \subset  \cup_{\ell<j}V^\ell_{i}
\label{Inductive:cd}\\
& \text{the pair } (T_{i-1},T_i) \text{ is a Gabai Link of order }  n_i.
\label{Inductive:e}\end{eqnarray}
\end{subequations} 

For the inductive step, let $h_{k+1}:T\rightarrow R^3$ be given by 
$h_k\circ g_{n_{k+1}}^{-1}$ and let $T_{k+1}=h_{k+1}(T)$,  $A_{k+1}=h_{k+1}(A)$,  $B_{k+1}=h_{k+1}(B)$, and $V^{j}_{k+1}=h_{k+1}(V^j)$. Note that the pair $(T_{k},T_{k+1})$ is then homeomorphic to $(T^{\prime}_{n_{k+1}},T)$ via the homeomorphism  
$g_{n_{k+1}}\circ h_k^{-1}$
and so forms a Gabai Link of index $2n_{k+1}$. This shows that Statement (\ref{Inductive:e}) holds when $i=k+1$. Properties (\ref{Inductive:ab}) and (\ref{Inductive:cd}) for $i=k+1$ follow by applying $h_{k+1}$ to properties 
(\ref{SetupConditions1}), (\ref{SetupConditions2}), and (\ref{SetupConditions3})
 from Lemma \ref{SetupLemma}. This completes the verification of the inductive step and shows that the following lemma holds.

\begin{lemma} \label{GabaiLemma} 
The Gabai 3-manifold $G = \ds \bigcup_{m = 0}^\infty T_m$
constructed  above satisfies the properties listed in  (\ref{Inductive:ab}),  
(\ref{Inductive:cd}), and (\ref{Inductive:e}) for all $i>0$.
\end{lemma}

\subsection{Main Result on Gabai Manifolds}
Using the notation from the previous subsection we can state and prove our main result about Gabai manifolds.

\begin{theorem}
Let  $G = \ds \bigcup_{m = 0}^\infty T_m$ be a Gabai contractible 3-manifold where each $T_m$ is a solid torus and consecutive tori form a Gabai Link.  Then $G$ satisfies the double 3-space property.\label{GabaiTheorem}
\end{theorem}

{\sc Proof.}  
The key to the proof is that in the Gabai manifold $G$, we may assume that the  conditions in Lemma \ref{GabaiLemma} are satisfied.

To show that $G$ satisfies the double 3-space property, we choose the closed sets $A^{\,\prime}= \ds \cup_{n=0}^\infty A_n$ and  $B^{\,\prime} = \ds \cup_{n=0}^\infty B_n$.  Recall that $A_n=h_n(A)=h_n(h(B^2\times C_1))$ and that $B_n=h_n(B)=h_n(h(B^2\times C_2))$. 
We claim that $M = G - A^{\,\prime}$, $N = G - B^{\,\prime}$ and $M \cap N = 
G - (A^{\,\prime} \cup B^{\,\prime})$ are each homeomorphic to $\RR^3$.  

We first show  $M \cap N = G - (A^{\,\prime} \cup B^{\,\prime})$ is homeomorphic to $\RR^3$.  It suffices to show that  $M\cap N$ is an increasing union 
of  copies of  $\RR^3$ \cite{Bro61}. First notice that ${\rm Int}\ V_n^0\subset T_{n}$ is homeomorphic to $\RR^3$ since it is the product of an open interval and an open 2-cell.  Next notice that $M\cap N = \ds \bigcup_{n=0}^\infty {\rm\  Int} \  V_n^0$ because any point $p$ in $M\cap N$ must lie in the interior of some $V_m^i$ and therefore lies in the interior of  $V_{m+i}^0$ by condition (3b) in Lemma \ref{GabaiLemma}. Again  by condition (3b) in Lemma \ref{GabaiLemma}, the ${\rm Int}\ V_n^0$ are nested. So $M\cap N$ is an increasing union of copies of $\RR^3$, and so is homeomorphic to $\RR^3$.

The proofs that $M$ and $N$ are homeomorphic to $\RR^3$ are similar, so we will just focus on $M$. Let 
$W^0=V^0\cup V^1\cup h(B^2\times[2/3,1])\subset T$ and let $W^i=V^{i+1}\cap h(B^2\times[0,1/3])\subset T$. Then $T-\bigcup_{i=0}^\infty W^i$ is precisely $ A$.
Let $h_0=h$ and $W^{n}_{i}=h_{i}(W^n)$. Then as in the preceding paragraph, by the conditions in Lemma \ref{GabaiLemma}, $M =\bigcup_{n=0}^\infty {\rm\  Int} \ W_0^n$ which is an increasing union of copies of $\RR^3$. So $M$ is homeomorphic to $\RR^3$.
\qed

\begin{corollary} 
There are uncountably many distinct contractible 3-manifolds with the double 3-space property.
\end{corollary}

{\sc Proof.}  This follows directly from Theorem \ref{GabaiTheorem} and the discussion following Definition \ref{GabaiManifoldDef}. \qed

\section{Interlacing Theory}\label{InterLacingSec}
We now work towards showing there is an infinite class  of contractible, genus one 3-manifolds that fail to have the double 3-space property. The key concept needed here is a generalization of interlacing theory, introduced in 
\cite{Wri89}.

\begin{definition}  
Let $A$ and $B$ be finite subsets of a simple closed curve $J$ each containing $k$ points. We say $(A,B)$ is a \emph{$k$-interlacing of points} if each component of $J-A$ contains exactly one point of $B$.
\end{definition}

The idea is that in traversing the curve one alternates $k$ times from the set $A$ to the set $B$.
Note that the definition also implies that each component of $J-B$ contains exactly one point of $A$.

\begin{definition}
Let $A$ and $B$ be disjoint compact sets.  We say that $(A,B)$ is a \emph{$k$-interlacing for a simple closed curve $J$} if there exist finite subsets $A' \subset A \cap J$ and $B' \subset B \cap J$ so that  $(A',B')$ is a $k$-interlacing of points, but it is impossible to find such subsets that form a $(k+1)$-interlacing of points.  If either $A \cap J = \emptyset$ or $B \cap J=\emptyset$, then we say that $(A,B)$ is a $0$-interlacing.
\end{definition}

\textbf{Note:}  If $(A,B)$ is a \emph{$k$-interlacing for a simple closed curve $J$}, and $(A\subset C, B\subset D)$, then $(C,D)$ is an \emph{$\ell$-interlacing for the simple closed curve $J$} where $\ell\geq k$.

\begin{theorem}[\textbf{Interlacing Theorem for a Simple Closed Curve}]  
If $A$ and $B$ are disjoint compact sets and $J$ is a simple closed curve, then $(A,B)$ is a $k$-interlacing for some non-negative integer $k$.
\end{theorem}

{\sc Proof.}  
If $A \cap J= \emptyset$ or  $B \cap J= \emptyset$, then  $(A,B)$ is a $0$-interlacing.  Otherwise, using compactness, it is possible to cover $A \cap J$ with a finite collection of non-empty, connected,  disjoint open sets $U_1, U_2, \dots , U_m$ and cover $B \cap J$ with a finite collection of non-empty, connected,  disjoint open sets $V_1, V_2, \dots , V_n$ so that the $U_i$ and $V_j$ are also disjoint.  If  $A' \subset A$ and $B' \subset B$ so that $(A',B')$ is a $k$-interlacing of points for $J$, then $A'$ contains at most one point from each $U_i$ and $B'$ contains at most one point from each $V_j$.  So there is a bound on $k$, and our theorem is proved. 
\qed

\begin{theorem}[\textbf{Neighborhood Interlacing Theorem for Simple Closed Curves}]  
If  $(A,B)$ is a $k$-interlacing for a simple closed curve $J$,  then there are open neighborhoods $U$  and $V$ of $A \cap J$ and $B \cap J$, respectively, in $J$ so that if $\widetilde{A}$ and $\widetilde{B}$ are disjoint compact sets with $A \cap J \subset \widetilde{A} \cap J \subset U$ and $B \cap J \subset  \widetilde{B} \cap J \subset V$, then $(\widetilde{A},\widetilde{B})$ is also a $k$-interlacing.
\end{theorem}

{\sc Proof.}   
As in the proof of the preceding theorem find the non-empty, connected,  disjoint open sets $U_i$ and $V_i$, but in addition we may assume that $m=n=k$.  Let $ U=\bigcup_{i=1}^m U_i$ and $V=\bigcup_{i=1}^n V_i$.
\qed

The following theorem helps to clarify the concept of a meridional disk with holes.

 \begin{theorem} [\textbf{Meridional Disk with Holes Theorem}] \cite[Theorem A6]{Wri89}
Let $H$ be a properly embedded disk with holes in a solid torus $T$.  Then $H$ is a meridional disk with holes if and only if the inclusion $f:H \to T$ lifts to  a map $\hat{f}$  from $H$ to the universal cover $\widetilde{T} =  B^2 \times \mathbb R$ and $\hat{f}(H)$  separates 
$\widetilde{T}$ into two unbounded components.  
\end{theorem} 

\begin{example}
Figure \ref{diskwHolesFig} illustrates a properly embedded disk with holes in a solid torus. Here M1, M2, and M3  represent meridional disks with two holes. T1 represents a tube joining a hole in M1 to the boundary of the solid torus. T2 and T3 represent tubes joining a hole in M1 to a hole in M2, and a hole in M2 to a hole in M3, respectively. T4 represents a tube joining a hole in M3 to the boundary of the solid torus.  Note that this disk with holes is topologically a planar disk with four  holes, that three of the boundary components are meridional curves in the solid torus, and that two of the boundary components are trivial curves in the boundary of the solid torus.
\end{example}

\begin{center}
\begin{figure}[ht]
 \includegraphics[width=0.5\textwidth]{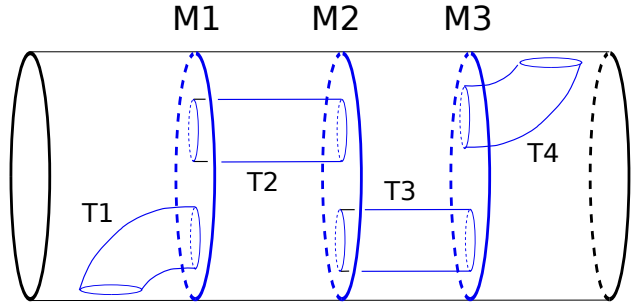}
\caption{%
Meridional Disk with Holes     }%
\label{diskwHolesFig}
\end{figure}
\end{center}

\begin{definition}  
Let $A_1, A_2, \dots, A_k, B_1, B_2, \dots, B_k$ be disjoint meridional disks with holes in a solid torus $T$.  Let $A = \cup_{i=1}^k A_i$ and $B = \cup_{i=1}^k B_i$.  We say that $(A,B)$ is a \emph{$k$-interlacing collection of meridional disks with holes} if each component of $T - A$ contains exactly one $B_i$.
\end{definition}

Note that the definition also implies that each component of $T-B$ contains exactly one  $A_j$.

\begin{definition}
Let $A$ and $B$ be disjoint compact sets.  We say that $(A,B)$ is a \emph{$k$-interlacing for a solid torus $T$} if there exist disjoint meridional disks with holes in $T$,   $A_1, A_2, \dots, A_k, B_1, B_2, \dots, B_k$ with $A' = \cup_{i=1}^k A_i \subset A$ and  $B' = \cup _{i=1}^kB_i \subset B$ so that $(A',B')$ is a $k$-interlacing collection of meridional disks with holes, but it is impossible to find such subsets that form a $(k+1)$-interlacing collection of meridional disks with holes.  If either $A$ or $B$ fails to contain a meridional disk with holes in $T$, then we say that $(A,B)$ is a $0$-interlacing.
\end{definition}

\begin{example}
 Figure \ref{InterlacingFig} illustrates interlacing for two different types of links. Note that for the Gabai Link, the two meridional disks A and B form a 1-interlacing both for the outer and inner component. For the other link (a McMillan Link introduced in the next section), the disks form a 1-interlacing for the outer torus, but form a 3-interlacing for the inner torus. That fact that the interlacing number always increases for the inner torus for these type of links is the key observation that will allow us to prove that certain manifolds constructed using these links do not have the double 3-space property.
\end{example}

\begin{figure}[ht]
\includegraphics[width=0.45\textwidth]{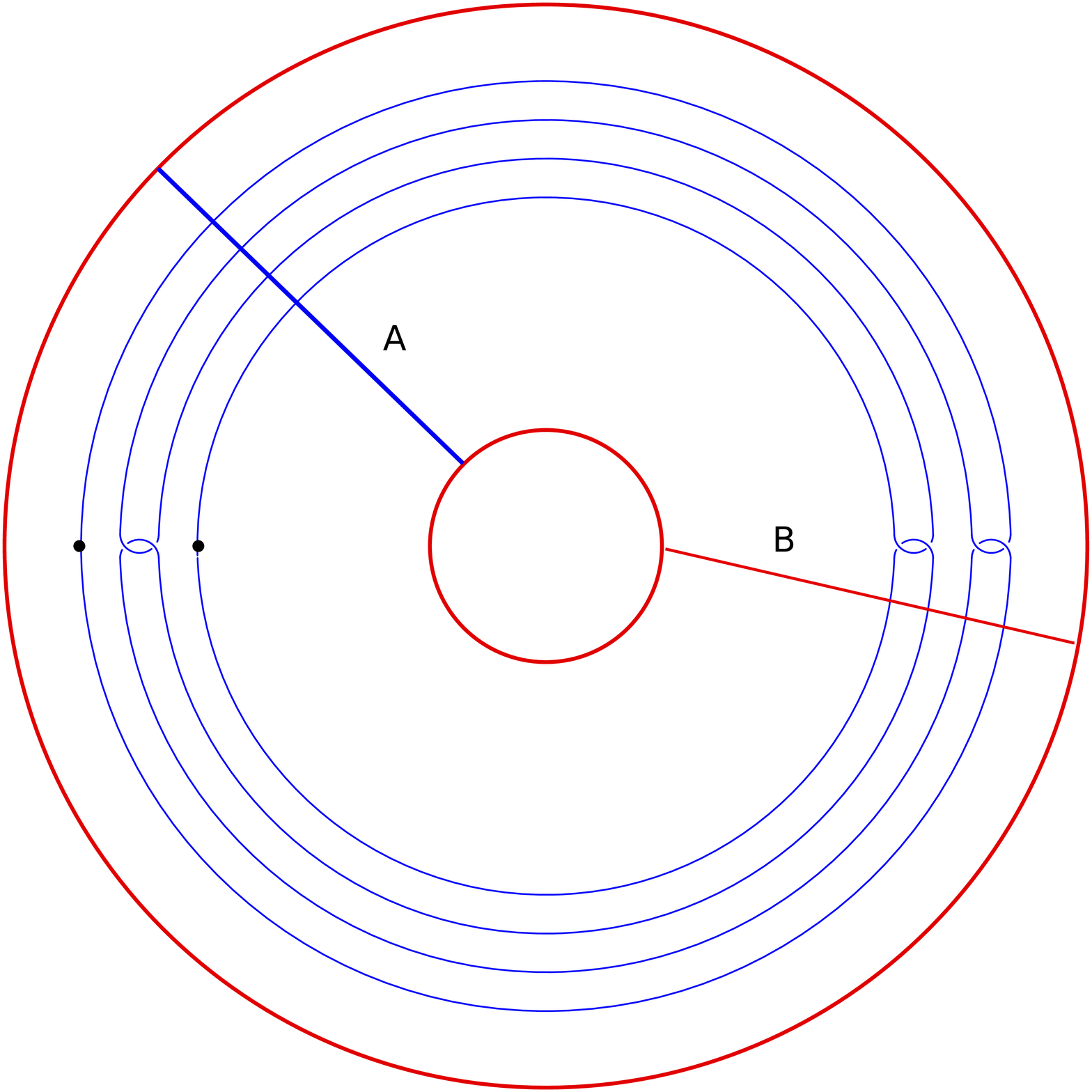}
\ 
\includegraphics[width=0.45\textwidth]{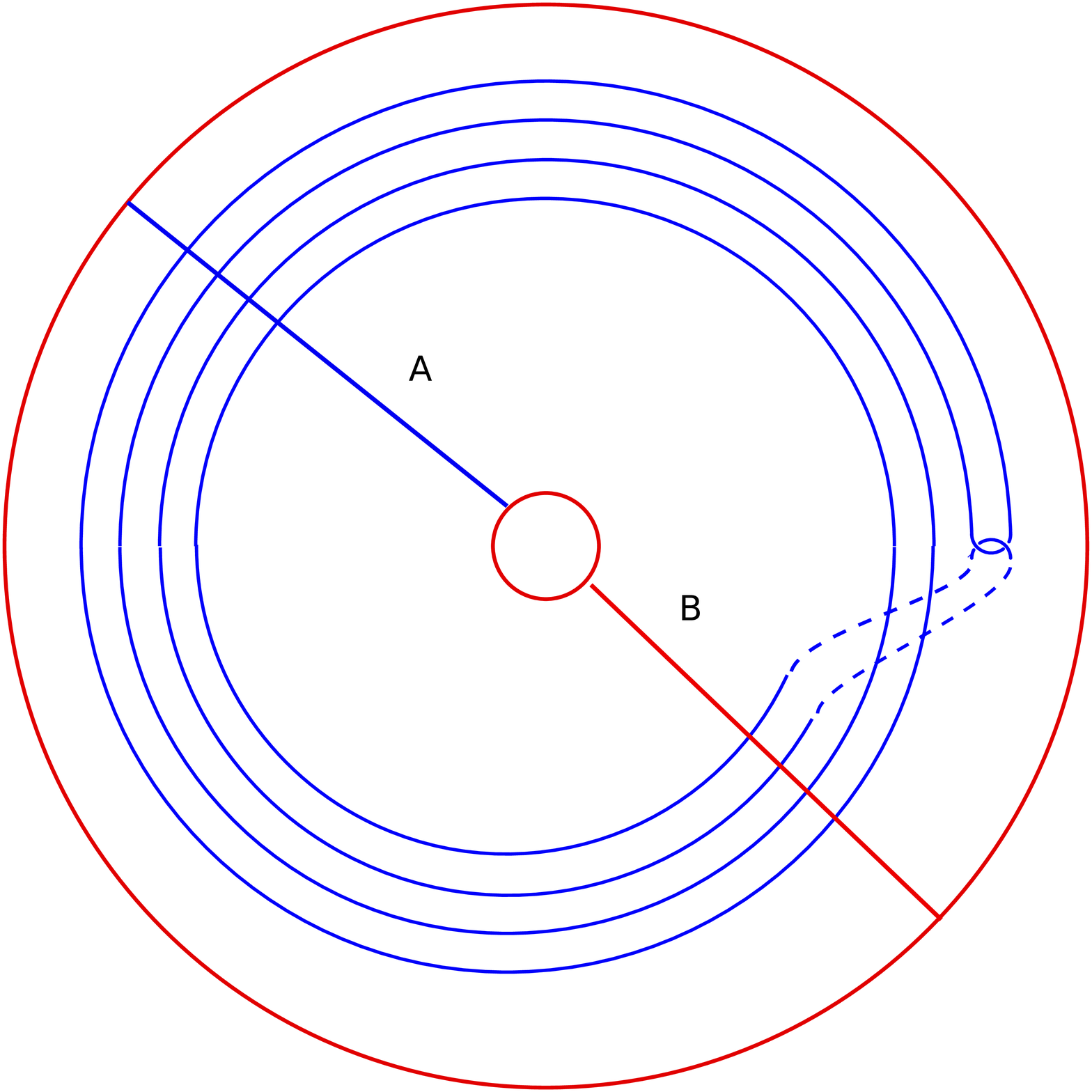}
\caption{%
Interlacing Examples   }%
\label{InterlacingFig}
\end{figure}
    
 A curve in a solid torus is \emph{essential} if it is nontrivial homotopically. 
 The \emph{standard core} for a solid torus $T=S^1\times B^2$ is the simple closed curve $S^1\times \{0\}$. A \emph{core} for the solid torus is any simple closed curve in the torus that is ambiently isotopic to the standard core, fixing the boundary. Equivalently, a core is a simple closed unknotted curve in the interior of  $T$ that has geometric index 1 in $T$ by the remarks after Definition \ref{geometric}.

 \begin{lemma} \label{important lemma}  
 If  $(A,B)$ is a  $k$-interlacing collection of meridional disks with holes for the solid torus $T$ and $J$ is a core for $T$, then $(A, B)$ is an $n$-interlacing of $J$ where $n \ge k$.
\end{lemma}

{\sc Proof.}  
If $k=0$ or $k=1$ the proof is quite easy.  Each component of $T - A$ contains exactly one meridional disk with holes component of $B$.  Let $J$ be a core for $T$.  Since each disk with holes component of $A$ is interior essential, $J$ must meet each component of $A$.  Let $U$ be a component of $J - A$ so that the endpoints of the closure of $U$ are in different components of $A$. Then $U$ must meet a component of $B$ since each component of $B$ is interior essential.  Since there are at least $k$ such components of $J - A$ with endpoints of the closure  in different components of $A$, $(A,B)$ must be at least a  $k$-interlacing for $J$.  Thus we see that  $(A, B)$ is an $n$-interlacing of $J$ where $n \ge k$.
 \qed

\begin{theorem}[\textbf{Interlacing Theorem for a Solid Torus}] 
If $A$ and $B$ are disjoint compact sets and $T$ is a solid torus, then $(A,B)$ is a $k$-interlacing  of $T$ for some non-negative integer $k$.
\end{theorem}

{\sc Proof.}   
We just need to show that the interlacing number of $(A,B)$ with respect to $T$ is bounded.  Let $J$ be a core of the solid torus.  The interlacing number of $(A,B)$ with respect to  $T$ is less than or equal to the interlacing number of  $(A,B)$ with respect to $J$ which is well-defined by the Interlacing Theorem for simple closed curves.  
\qed

\section{McMillan Contractible 3-Manifolds Do Not Satisfy the Double 3-space Property}
\label{McMSec2}
There is an alternative generalization of a Whitehead Link that was used by McMillan  \cite{McM62} to show the existence of uncountably many contractible 3-manifolds.  We call these links \emph{McMillan Links}.  

\begin{definition}  
Let $n$ be a positive integer.  A \emph{McMillan Link} of order $n$ is a pair of solid tori $T' \subset T$ so that $T'$ is embedded in $T$ as illustrated    in Figure  \ref{McLinkFig} for a McMillan Link of order 2 and of order $n$. 
\end{definition}

\textbf{Note:} A McMillan Link of order $n$ can be constructed as follows. In an unknotted solid torus $T$ place in the interior a solid torus $T_1$ that has winding number $n$. Then place a solid torus $T_2$ in $T_1$ so that the pair $(T_1, T_2)$ forms a Whitehead Link.
The pair $T_2\subset T$ is then a McMillan Link of order $n$. Since geometric index is multiplicative, the geometric index of  the interior torus in a McMillan Link of order $n$ is seen to be $2n$ in the outer torus of the link.

\begin{definition}   
If $M$ is a genus one 3-manifold with defining sequence $(T_i)$, then we say that $M$ is a \emph{McMillan contractible 3-manifold} if for each $i$, $T_i \subset T_{i+1}$ is a McMillan Link  of order at least 2.
\end{definition}

\begin{figure}[ht!]
     \begin{center}
        \subfigure[McMillan Link: Order 2 and Index 4]{%
            \label{figM:first}
            \includegraphics[width=0.49\textwidth]{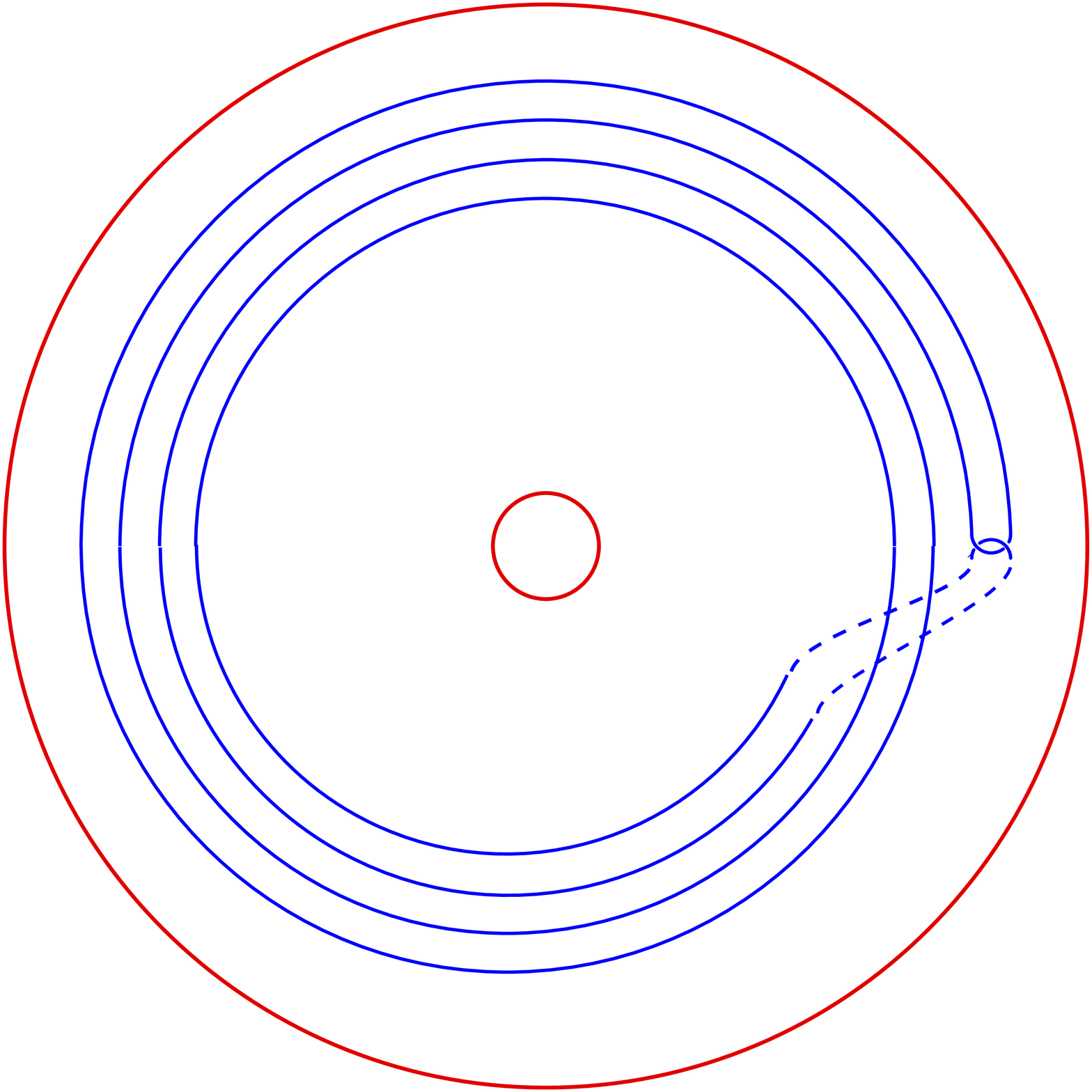}
        }%
        \subfigure[McMillan Link: Order n and Index 2n]{%
           \label{figM:second}
           \includegraphics[width=0.49\textwidth]{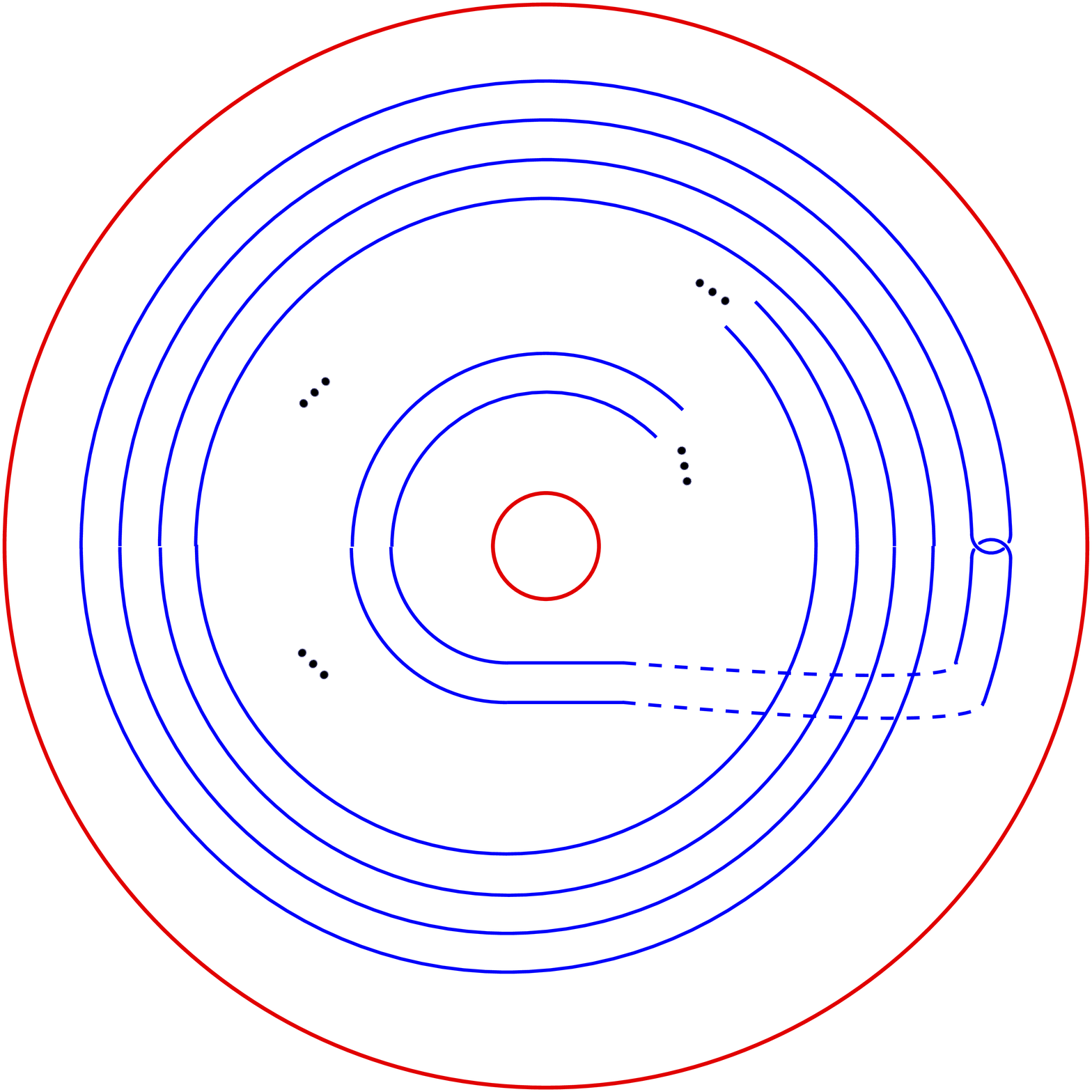}
        }\\ 
    \end{center}
    \caption{%
        McMillan Links     }%
   \label{McLinkFig}
\end{figure}

There are immediate results that follow from the previous section.

\begin{theorem}[\textbf{Interlacing Theorem for a McMillan Link}]  Suppose that $A$ and $B$ are disjoint planar 2-manifolds properly embedded in a solid torus $T$ so that $(A,B)$ is a $k$-interlacing for $T$.  If \  $T'$ is a McMillan Link of order  $n$ in $T$ so that  $T'$ is in general position with respect to $A \cup B$, then  $(A ,B)$ is an $m$-interlacing for $T'$ where $m \ge 2nk -1$.  
\end{theorem}

{\sc Proof.} 
Let $p: \widetilde{T} \to T$ be the projection map from the $n$-fold cover of $T$.  Since $(A,B)$ is a $k$-interlacing for $T$, there exist disjoint meridional disks with holes, $A_1, A_2, \dots, A_k$ and $B_1, B_2, \dots, B_k$ with $ A' = \cup_{i=1}^k A_i \subset A$ and $ B' = \cup_{i=1}^k B_i \subset B$ so that $(A',B')$ is a $k$-interlacing collection of  meridional disks with holes for $T$.  Set $\widetilde{A}' = p^{-1}(A')$ and  $\widetilde{B}' = p^{-1}(B')$.  Using the Meridional Disk with Holes Theorem, we see that  $(\widetilde{A}',\widetilde{B}')$ is an $nk$-interlacing collection of meridional disks with holes for $\widetilde{T}$.  Let $\widetilde{i} : T' \to \widetilde{T}$ be a lift of the inclusion map $i : T' \to T$. Then $T'' =\widetilde{i}(T')$ is a Whitehead Link in $\widetilde{T}$. By \cite[Lemma A10]{Wri89}  $(\widetilde{A}',\widetilde{B}')$ is an $m$-interlacing of $T''$ where $m \ge 2nk-1$. It now follows that $(A,B)$ is an $m$-interlacing for $T'$ for $m \ge 2nk-1$.
\qed

\textbf{Note:} 
To see a simple illustration of this theorem, consider the second figure in Figure \ref{InterlacingFig}. $(A,B)$ is a 1-interlacing for the outer torus $T$. The inner torus $T^\prime$ is a McMillan Link of order 2 in $T$. In the double cover of $T$, $T^\prime$ lifts to a Whitehead Link. Note that $(A,B)$ is a 3-interlacing for $T^\prime$.

\begin{corollary} \label{corollary}   
In the previous theorem, if 
$k\geq 1$, and the McMillan Link $T^{\prime}\subset T$ is of order at least 2, then $m > k$. \end{corollary}

We now prove some lemmas that are needed in proving that McMillan contractible 3-manifolds do not have the double 3-space property.

\begin{lemma} \
Let $H$ be a properly embedded 2-manifold in a solid torus $T$ so that each component of $H$ is an interior-inessential disk with holes.  Then there is an essential simple closed curve in $T$ that misses $H$.
\end{lemma}

{\sc Proof.} 
Let $J$ be an oriented essential simple closed curve in $T$ that is in general position with respect to $H$.  We assume that $T$ lies in $\mathbb{R}^3$ . The proof is by induction on the number of points in $J \cap H$.  Consider a component $H'$ of $H$ that meets $J$ and let $z=\partial H^\prime$.  Since $H^\prime$ is interior-inessential, $z$ must be homologically trivial in $R^3-J$. Choose an orientation on $H'$.  Since $z$ is homologically trivial in $R^3-J$, the algebraic intersection number of $J$ and $H'$ is zero (meaning that there are the same number of positive and negative intersections). See for example the material on intersection numbers in \cite{Dol95} or \cite{SeTh80}.

 Let $p, q \in J \cap H'$ be points with opposite orientations.  The points $p$ and $q$ separate $J$ into two components $J_1$ and $J_2$.  Let $A$ be an arc in $H'$ between $p$ and $q$ that misses all other points of $J \cap H'$.  Then $J_1 \cup A$ and $J_2 \cup A$ are simple closed curves.  If $J_1 \cup A$ and $J_2 \cup A$ are both inessential in $T$, then so is $J$, so at least one of  $J_1 \cup A$ and $J_2 \cup A$ is essential in $T$.  We suppose that $J' =  J_1 \cup A$ is essential in $T$.  Using a collar on $H'$, we can push $J_1$ off $H$ to get an essential simple closed curve $J''$ that meets $H$ in two fewer points than $J$.
\qed

\begin{lemma} \label{meridional disk with holes}
Let $M$ be a 3-manifold so that $M = U \cup V$ where $U$, $V$ are homeomorphic to $\mathbb R^3$.  Let $T \subset M$ be a solid torus so that for every essential simple closed curve $J \subset T$, $J \not\subset U$ and $J \not\subset V$.   Let $C = M - U$ and $D = M - V$.  Then any neighborhood of  $T \cap C$ in $T$  contains a meridional disk with holes.
\end{lemma}

{\sc Proof.} 
Notice that by DeMorgan's Law, $C \cap D = \emptyset$.  Since $T \not\subset U$ and $T \not\subset V$, then $C' =T \cap C \ne \emptyset$ and $D' = T \cap D \ne \emptyset$. So $C'$ and $D'$ are disjoint non-empty compact subsets of $T$.  Let $N$ be an open neighborhood of $C'$ in $T$ that misses $D'$.
Let $K = T - N$.  Then $K$ is a compact set in $U$ that contains $D'$.  Since $U$ is homeomorphic to $\mathbb R^3$, $K$ is contained in the interior of a 3-ball $B \subset U$ with boundary a 2-sphere $S$ that we may suppose is in general position with respect to $T$.  Notice that $C'$ and $D'$ are in separate components of $M-S$ and so $S \cap T = \emptyset$ is impossible.  

Also, $S \subset {\rm Int}\ T$ is impossible because this would allow for an essential simple closed curve in $T$ that would lie in either $U$ or $V$.  
To see this, note that $S \subset {\rm Int}\ T$ implies that $S$ bounds a 3-cell $B^{\,\prime}$ in $T$. 
Either this 3-cell is equal to $B$ and contains $D^{\,\prime}$, or it is the complement in $M$ of $\text{int}(B)$ and contains $C^{\,\prime}$. In either case, a longitudinal essential curve in $T$ misses 
$C^{\,\prime}$ or $D^{\,\prime}$ and is thus contained in $U$ or $V$ which can't happen.

Thus the 
 set  $H = S \cap T \ne \emptyset$ lies in the neighborhood $N$ of $C'$, and each component of $H$ is a disk with holes.  If each component is interior-inessential, then, by the previous lemma, there is an essential simple closed curve $J$ in $T$ that misses $H$.  So $J$ lies in a component of $M-S$ and must miss either $C$ or $D$.  So $J \subset U$ or $J \subset V$ which is a contradiction.  Thus at least one of the components of $H$ must be interior-essential and thus a meridional disk with holes.
\qed

\begin{theorem}\label{McMTheorem}  No McMillan contractible 3-manifold $M$ can be expressed as  the union of two copies of $\mathbb R^3$.
\end{theorem}

{\sc Proof.} 
Let $T_i$ be a defining sequence for $M$ so that  $\displaystyle M = \cup_{i=0}^\infty T_i$ .
Suppose $M = U \cup V$ where $U$, $V$ are homeomorphic to $\mathbb R^3$.  Then by Theorem \ref{essential}, for each essential simple closed curve $J' \subset T_i$, $J' \not\subset U$ and $J' \not\subset V$.  Let $C = M - U$ and $D = M - V$. Then by Lemma \ref{meridional disk with holes}, for each $i$, each neighborhood of $T_i \cap C$ in $T_i$ and each neighborhood of $T_i \cap D$ in $T_i$ contains a meridional disk with holes for $T_i$.

Let $J$ be a simple closed curve core of $T_0$.  Let $n$ be the interlacing number of $(J \cap C, J \cap D)$. Let $\widetilde{C}$ and $\widetilde{D}$ be closed neighborhoods in $J$ of  $J \cap C$ and $J \cap D$, respectively so that the interlacing number for  $(\widetilde{C},\widetilde{D})$ is also $n$.   Let $H_C$ be a meridional disk with holes in a neighborhood of $C \cap T_n$  in $T_n$  and $H_D$ be a meridional disk with holes in a neighborhood of $D \cap  T_n$  in $T_n$ so that

\begin{enumerate}
\item  $H_C \cap H_D = \emptyset$
\item  $H_C \cap J \subset \widetilde{C}$,  $H_D \cap J \subset \widetilde{D}$
\item $H_C$ and $H_D$ are in general position with respect to $T_i, 0 \le i \le n$.
\end{enumerate}

The interlacing number of
 $(H_C, H_D )$ in $T_{n}$
 is at least one since both intersections are nonempty. By Corollary \ref{corollary} the interlacing number of 
$(H_C, H_D)$ in $T_{n-j}$ is greater than $j$.  In particular, when $j=n$, this implies that the interlacing number of 
$(H_C, H_D)$ in $T_0$ is greater than $n$.
So for some $k>n$, there are points $c_1,\cdots, c_k$ in $(H_C \cap J)$ and points
$d_1,\cdots, d_k$ in $(H_D \cap J)$ that form a $k$-interlacing in $J$. These same points show that the interlacing number of  $(\widetilde{C},\widetilde{D})$ in $J$ is also greater than $n$, a contradiction to the earlier statement that this interlacing number is exactly $n$.

Thus $M$ can not be expressed as the union of two homeomorphic copies of $R^3$.
\qed

\begin{corollary}
There are uncountably many distinct contractible 3-manifolds that fail to have the double 3-space property.
\end{corollary}

{\sc Proof.}  
This follows directly from Theorem \ref{McMTheorem} and the discussion following Definition \ref{GabaiManifoldDef}. 
\qed

\section{Questions and Acknowledgments}

The results in this paper produce two infinite classes of genus one contractible 3-manifolds, one of which has the double 3-space property and one of which does not. There are many genus one contractible 3-manifolds that do not fit into either of these two classes. This leads to a number of questions. 
\begin{question} Is it possible to characterize which genus one contractible  
3-manifolds have the double 3-space property?
\end{question}
\begin{question} Is it possible to characterize which  contractible 3-manifolds have the double 3-space property?
\end{question}
\begin{question} Is there a contractible 3-manifold M which is the union of two copies of $\RR^3$, but which does not have the double 3-space property?
\end{question}

The first and second authors were supported in part by the Slovenian Research Agency grant BI-US/15-16-029. 
The first author was supported in part by the National Science Foundation grant DMS0453304.
The first and third authors were supported in part by the National Science Foundation grant DMS0707489. 
The second author was supported in part by the Slovenian Research Agency grants P1-0292,  J1-8131, and J1-7025.  The authors would also like to thank the referee for a number of helpful suggestions on clarifying our presentation.

\section{Appendix: Geometric Index of Gabai Links}
We show in this section in a Gabai Link of order $n$, the inner torus has geometric index $2n$ in the outer torus. We use a generalization of an argument introduced by Andrist and Wright in 
\cite{AW00}. Once we set up the notation, the argument is relatively direct.

\textbf{Setup:} Consider a Gabai Link $T_2\subset T_1$ of order $n$. Using the parameterization of the $S^1$ factor of $T_1$ as in Section \ref{SpecialSubsets},
place meridional disks $M1$ and $M2$ corresponding to parameters $x=1/6$ and
$x=5/6$.  This is illustrated in Figure \ref{fig-G1:first} for  a Gabai Link of order 3. Note that each meridional disk intersects the standard core of $T_2$ in $2n$ points.

For any simple closed curve $C$ in $M1$ or $M2$ that misses the standard core of $T_2$, let $C^{\prime}$ be the subdisk in $M1$ or $M2$ bounded by $C$ and let $o(C)$ be the number of points of intersection of  $C^{\prime}$ with the standard core of  $T_2$.
Notice that the geometric index of a torus in a larger torus is the same as the geometric index of any core of the interior torus in the larger torus,  so we  use the standard core.

A \emph{Bing Link} in $T_1$ is a pair of tori linked as in Figure \ref{fig-BW:second}. The same argument as in Remark \ref{WH-index} shows that the geometric index of a Bing Link in the outer torus is 2.

\begin{figure}[h!]
     \begin{center}
        \subfigure[Gabai Link with Meridional disks]{%
            \label{fig-G1:first}
            \includegraphics[width=0.5\textwidth]{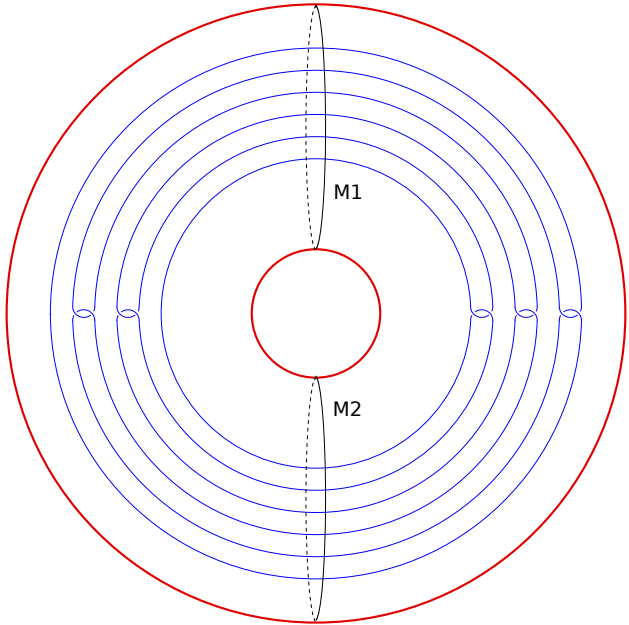}
        }%
        \subfigure[Link with Neighborhoods Removed]{%
           \label{fig-G1:second}
           \includegraphics[width=0.5\textwidth]{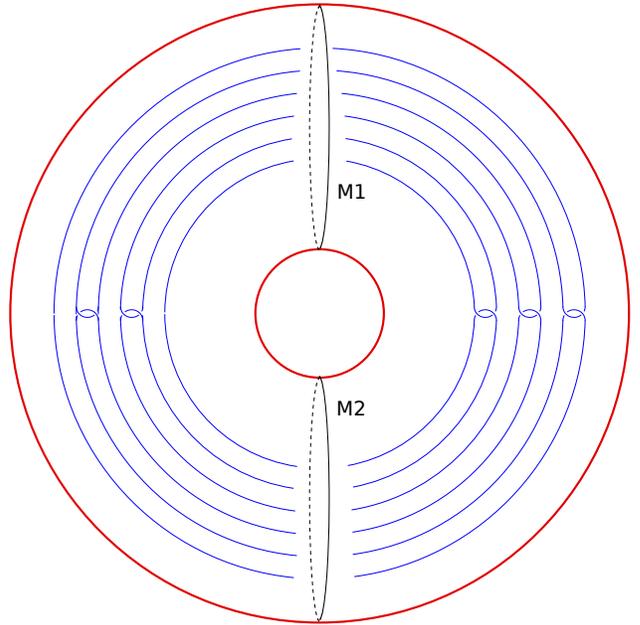}
        }\\ 
        \subfigure[Bing and Whitehead Links]{%
            \label{fig-BW:first}
            \includegraphics[width=0.5\textwidth]{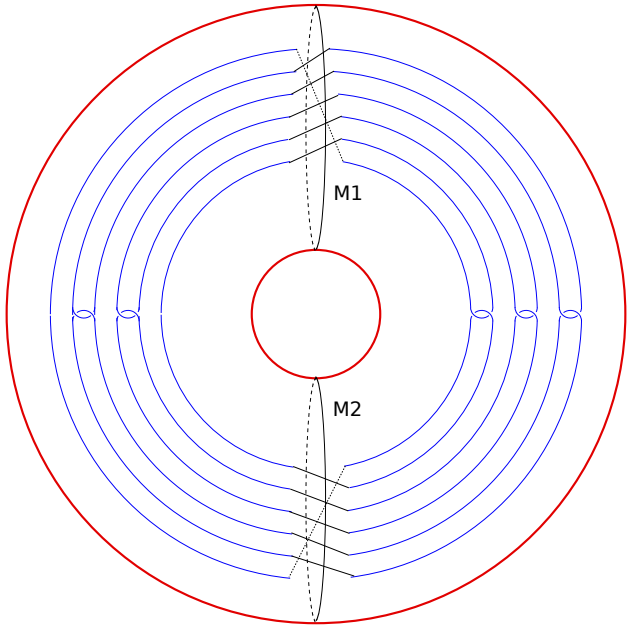}
        }%
        \subfigure[Bing Link]{%
           \label{fig-BW:second}
           \includegraphics[width=0.48\textwidth]{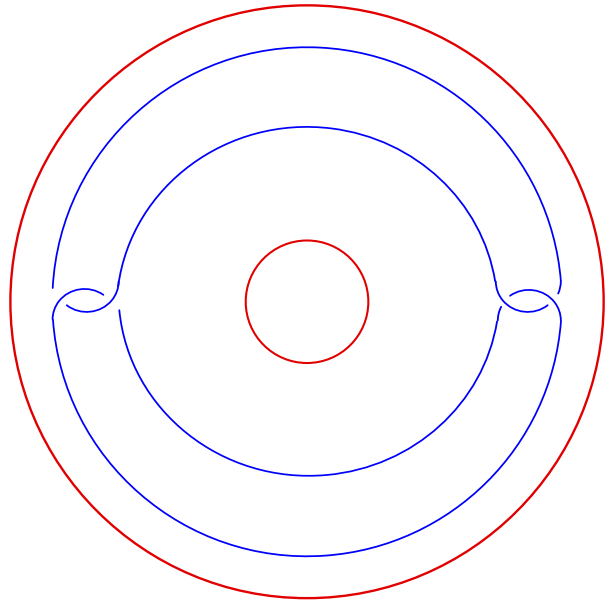}
        }\\ 
    \end{center}
    \caption{%
        Illustrating  the Geometric Index of  a Gabai Link     }%
   \label{GindexFig}
\end{figure}

\begin{lemma}\label{G index 1}
Any meridional disk in $T_1$ missing M1 and M2 must intersect any core of $T_2$ in at least 2n points. 
\end{lemma}
\begin{proof}
Such a meridional disk lies in the complement of a neighborhood $N$ of M1 and M2. Form a collection $T_3$ of $n-1$ Bing Links and one Whitehead Link by taking $T_2-N$ and connecting the tubes on each side of $M1$ and $M2$ as in Figure \ref{fig-BW:first}.  Note that the  pair $(T_1 - N, T_2 -N)$ is homeomorphic to the pair  $(T_1 - N, T_3 -N)$.  Since the geometric index of Bing Links and Whitehead Links in $T_1$ is 2, the meridional disk $M$ intersects any core of $T_2$ in at least $2n$ points.
\end{proof}

\begin{lemma}\label{G index 2}
Let D be a disk with boundary  $C$ in the interior of $T_1$ so that $D\cap (M1\cup M2)=C$ and so that $C$ misses the standard core of $T_2$. 
Then $D$ intersects the standard core of $T_2$ in at least $o(C)$ points.
\end{lemma}
\begin{proof}
Without loss of generality, assume $C\subset M1$. If $D$ intersected the standard  core of $T_2$ in fewer than $o(C)$ points,
$D_1=(M1-C^{\prime}) \cup D$ (where $C^\prime$ is as in the setup before Lemma 7.1) would intersect the standard core of $T_2$ in fewer than $2n$ points. This disk $D_1$ could be pushed slightly off $M1$, preserving intersections so as to contradict Lemma \ref{G index 1}.
\end{proof}

\begin{lemma} \label{G index 3}
In a Gabai Link of order n, the inner torus has geometric index 2n in the outer torus.
\label{gabaiindex}
\end{lemma}
\begin{proof}
The geometric index is  $\leq2n$ since both $M1$ and $M2$ intersect the standard core of $T_2$ in $2n$ points. So it suffices to show that the geometric index is  $\geq 2n$.

Let $D$ be a meridional disk in $T_1$ that intersects the standard core of $T_2$ in $k$ points where $k$ is minimal. Without loss of generality, $D$ is in general position with respect to $M1\cup M2$ and $\partial D \cap \partial (M1\cup M2)=\emptyset$. Among all such disks $D$, choose one so that the number of simple closed curves in $D\cap (M1\cup M2)$ is minimal. If there are no curves of intersection, Lemma \ref{G index 1} implies that $k\geq 2n$. 

If there are curves of intersection, choose an innermost (on D) such curve $C$. The disk $D^{\prime}$ in $D$ bounded by $C$ must intersect the standard core of $T_2$ in at least $o(C)$ points by Lemma \ref{G index 2}. $D^\prime$ can be replaced by $C^\prime$ without increasing the number of intersections with the standard core of $T_2$. In the new (possibly singular) disk $(D-D^{\prime}) \cup C^{\prime}$, any singularities  consist of disjoint double curves. These double curves can be cut apart and reattached so as to form a nonsingular disk $D^{\prime \prime}$ preserving the number of intersections with the standard core. See \cite[p. 42]{He04} for a discussion of double curves. The disk $D^{\prime\prime}$ can then be pushed slightly off $M1\cup M2$ thus decreasing the number of curves of intersection, contradicting the assumption that the number of such curves is minimal.
\end{proof}

\end{document}